\documentclass[12pt,reqno,twoside]{amsart}
\usepackage{graphicx}
\usepackage{amssymb}
\usepackage{epstopdf}
\usepackage[asymmetric,top=3.5cm,bottom=4.3cm,left=3.1cm,right=3.1cm]{geometry}
\usepackage{bm}

\usepackage{booktabs} 
\usepackage{array} 
\usepackage{paralist} 
\usepackage{verbatim} 
\usepackage{subfig} 
\usepackage{tabularx}
\usepackage{amsmath,amsfonts,amsthm,mathrsfs,amssymb,cite}
\usepackage[usenames]{color}

\newtheorem{thm}{Theorem}[section]

\newtheorem{lem}{Lemma}[section]

\theoremstyle{definition}
\newtheorem{defn}{Definition}[section]
\theoremstyle{remark}

\newtheorem{rem}{Remark}[section]
\numberwithin{equation}{section}

\usepackage{hyperref}
\hypersetup{
  colorlinks   = true,    
  urlcolor     = blue,    
  linkcolor    = blue,    
  citecolor    = red      
}
\usepackage{stmaryrd} 
\usepackage[disable,textwidth=1in,textsize=tiny]{todonotes}

\title[Nearly Non-scattering Electromagnetic Wave Set]{Nearly Non-scattering Electromagnetic Wave Set and Its Application}

\author{Hongyu Liu}
\address{Department of Mathematics, Hong Kong Baptist University, Kowloon Tong, Hong Kong SAR.\vspace*{-2mm}}
\address{and\vspace*{-2mm}}
\address{HKBU Institute of Research and Continuing Education, Virtual University Park, Shenzhen, P. R. China.}
\email{hongyu.liuip@gmail.com; hongyuliu@hkbu.edu.hk}

\author{Yuliang Wang}
\address{Department of Mathematics, Hong Kong Baptist University, Kowloon Tong, Hong Kong SAR.}
\email{yuliang@hkbu.edu.hk}

\author{Shuhui Zhong}
\address{Department of Mathematics, Tianjin University, Tianjin, P. R. China.}
\email{shuhuizhong@126.com}

\begin{document}
\maketitle

\begin{abstract}

For any inhomogeneous compactly supported electromagnetic (EM) medium, it is shown that there exists an infinite set of linearly independent electromagnetic waves which generate nearly vanishing scattered wave fields. If the inhomogeneous medium is coated with a layer of properly chosen conducting medium, then the wave set is generated from the Maxwell-Herglotz approximation to the interior PEC or PMC eigenfunctions and depends only on the shape of the inhomogeneous medium. If no such a conducting coating is used, then the wave set is generated from the Maxwell-Herglotz approximation to the generalised interior transmission eigenfunctions and depends on both the content and the shape of the inhomogeneous medium. We characterise the nearly non-scattering wave sets in both cases with sharp estimates. The results can be used to give a conceptual design of a novel shadowless lamp. The crucial ingredient is to properly choose the source of the lamp so that nearly no shadow will be produced by the surgeons operating under the lamp. 

\medskip

\medskip



\end{abstract}

\section{Introduction}

\subsection{Background and practical motivation}

Invisibility cloaking has received significant attentions in recent years in the scientific community due to its practical importance. The articles \cite{GLU,GLU2,Leo,PenSchSmi} pioneer the study on invisibility cloaking by metamaterials. The crucial idea is to coat a target object with a layer of artificially engineered material with desired optical properties so that the electromagnetic waves pass through the device without creating any shadow at the other end; namely, invisibility cloaking is achieved. There are many subsequent developments on various invisibility cloaking schemes and we refer to \cite{Ammari2,Ammari3,Ammari4,BL,BLZ,CC,GKLU4,GKLU5,HuLiu,KOVW,KSVW,LiLiuRonUhl,Liu,LiuSun,Nor,U2} and the references therein. All of the aforementioned works are concerned with the design of certain artificial mechanisms of controlling wave propagation in order to achieve the invisibility effect. Most of the invisibility cloaking results are independent of the source of the detecting waves; that is, for any generic wave fields that one uses to impinge on the cloaking device, there will be invisibility effect produced. We are also aware of the cloaking scheme in \cite{LiLiuRonUhl} where the invisibility depends on the incident angles of the impinging wave fields. 

In this paper, we are curious about the invisibility without the metamaterial coating; that is, whether or not invisibility be achieved for a regular/natural scattering object. It turns out that invisibility can still be nearly achieved for a large class of special impinging wave fields, depending on the underlying scattering object. Our study is motivated by a recent article \cite{BPS} where the authors consider the so-called non-scattering energy for the Schr\"odinger equation. It is shown that for a generic potential supported in a corner, there does not exist non-scattering energy; that is, for any incident wave the corresponding scattered wave cannot vanish outside the potential. Indeed, the only known case that there exist non-scattering energies is for radially symmetric potentials. Hence, it seems unobjectionable to conjecture that for a generic potential, there does not exist non-scattering energy. However, our study in this paper shall indicate that for any potential satisfying a certain generic condition, there always exist nearly non-scattering energies. Motivated by some practical applications, we shall conduct our study for the electromagnetic wave scattering in a rather different, but more general context. It is shown that in three different scenarios, for a certain inhomogeneous scattering object, there exists an infinite discrete set of ``energies", namely wavenumbers, so that the corresponding scattered wave fields are nearly vanishing due to certain incident wave fields. Hence, for those incident wave fields, near-invisibility is achieved. For the first two scenarios in our study,  we consider the case that the inhomogeneous medium is coated with a layer of properly chosen conducting medium, then the wave set is generated from the Maxwell-Herglotz approximation to the interior PEC or PMC eigenfunctions and depends only on the shape of the inhomogeneous medium. If no such a conducting coating is used, then the wave set is generated from the Maxwell-Herglotz approximation to the generalised interior transmission eigenfunctions and depends on both the content and the shape of the inhomogeneous medium. We characterise the nearly non-scattering wave sets in all of the three cases with sharp estimates. As an interesting application, the results can be used to give a conceptual design of a novel shadowless lamp. We shall briefly discuss the conceptual design of such a novel shadowless lamp in the following. Before that, we would like to note that the construction of ``non-scattering" wavenumbers with respect to a discrete set of measurement data in acoustic scattering governed by the Helmholtz equation was recently explored in \cite{BCN}; and the non-existence of ``non-scattering" wavenumbers was used to establish the uniqueness in determining an inhomogeneous acoustic medium scatterer in \cite{HSV}. Those studies are closely related to our present one. 

The shadowless lamp is an important medical device to provide illumination for surgical operations and medical examination. The lamp is fit for lighting needs in hospitals and the illumination can be adjusted according to practical requirements. It adopts light sources/reflectors from different positions to reduce shadows produced by different parts of the medical workers; see Fig.~\ref{fig:1} for an illustration.  
\begin{figure}[h]
\begin{center}
  \includegraphics[width=2in]{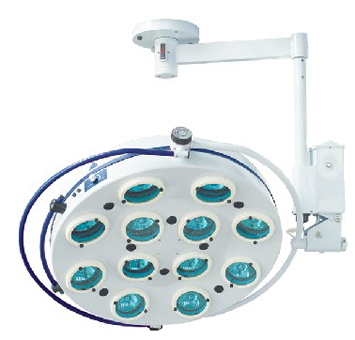}
  \end{center}
  \caption{Shadowless Lamp; www.amismed.com/pro$\empty_{-}$info.asp?pro$\empty_{-}$id=99  \label{fig:1}}
\end{figure}

Based on our study of the non-scattering electromagnetic wave set, we can readily propose a mathematical design of a novel shadowless lamp. The crucial idea is to choose the illumination sources of the lamp from a certain special set, depending on the surgeons performing the surgical operations under the lamp. Based on such a design, there would be nearly no shadow produced. The special set of illumination sources consists of the electromagnetic spectrum from low frequencies to high frequencies, and can be adjusted catering to the practical scenarios. If the surgeons wear certain coats made of suitably chosen conducting materials, then the design of the illumination sources depends only on the shapes of the surgeons; whereas if no such coats will be wore then the design of the sources depends on both the shapes and the optical properties of the bodies of the surgeons.   

\subsection{Mathematical formulation}

We consider the time-harmonic electromagnetic (EM) wave scattering in a homogeneous space with the presence of an inhomogeneous scatterer. Let us first characterise the optical properties of an EM medium with the electric permittivity $\epsilon$, magnetic permeability $\mu$ and electric conductivity $\sigma$. We recall that $\mathbb{M}_{sym}^{3\times 3}$ is the space of real-valued $3\times 3$ symmetric matrices and that, for any open set $D\subset\mathbb{R}^3$, we say that $\gamma$ is a tensor in $D$ satisfying the uniform ellipticity condition if $\gamma\in L^\infty(D; \mathbb{M}_{sym}^{3\times 3})$ and there exists $0<c_0<1$ such that
\begin{equation}\label{eq:elliptic}
c_0\|\xi\|^2\leq \gamma(\mathbf{x})\xi\cdot\xi\leq c_0^{-1}\|\xi\|^2\quad \mbox{for a.e. $\mathbf{x}\in D$ and every $\xi\in\mathbb{R}^3$}.
\end{equation}
$c_0$ shall be referred to as the {\it ellipticity constant} of the tensor $\gamma$. Moreover, $\gamma(\mathbf{x})$, $\mathbf{x}\in D$, is said to be {\it isotropic} if there exists $\alpha(x)\in L^\infty(D; \mathbb{R})$ such that $\gamma(\mathbf{x})=\alpha(\mathbf{x})\cdot\mathbf{I}_{3\times 3}$, where $\mathbf{I}_{3\times 3}$ signifies the $3\times 3$ identity matrix. It is assumed that both $\epsilon(\mathbf{x})$ and $\mu(\mathbf{x})$, $x\in\mathbb{R}^3$, belong to $L^\infty(\mathbb{R}^3;\mathbb{M}_{sym}^{3\times 3})$, and are uniform elliptic with constant $c_0\in\mathbb{R}_+$; whereas it is also assumed that $\sigma\in L^\infty(\mathbb{R}^3;\mathbb{M}_{sym}^{3\times 3})$ satisfying  
\begin{equation}\label{eq:elliptic2}
0\leq \sigma(\mathbf{x})\xi\cdot\xi\leq \lambda_0\|\xi\|^2\quad \mbox{for a.e. $\mathbf{x}\in \mathbb{R}^3$ and every $\xi\in\mathbb{R}^3$},
\end{equation}
where $\lambda_0\in\mathbb{R}_+$. Furthermore, we assume that there is an open bounded set $\Sigma$ with a Lipschitz boundary $\partial\Sigma$ and a connected complement $\Sigma^c:=\mathbb{R}^3\backslash\overline{\Sigma}$ such that 
\begin{equation*}
\epsilon(\mathbf{x})=\epsilon_\infty\cdot\mathbf{I}_{3\times 3};\quad \mu(\mathbf{x})=\mu_\infty\cdot\mathbf{I}_{3\times 3};\quad\sigma(\mathbf{x})=0\cdot\mathbf{I}_{3\times 3}\quad\mbox{for}\ \ \mathbf{x}\in\Sigma^c,
\end{equation*}
where $\epsilon_\infty$ and $\mu_\infty$ are two positive constants. $\epsilon_\infty$ and $\mu_\infty$ characterise the permittivity and permeability of the isotropic homogeneous matrix $\Sigma^c$ that contains the inhomogeneous scatterer $(\Sigma; \epsilon, \mu, \sigma)$. 

Let $\omega\in\mathbb{R}_+$ denote an electromagnetic (EM) wavenumber, corresponding to a certain EM spectrum. Consider the EM radiation in this frequency regime in the space $(\mathbb{R}^3;\epsilon,\mu,\sigma)$ described above. Let $(\mathbf{E}^{i,\omega}, \mathbf{H}^{i,\omega})$ be a pair of entire electric and magnetic fields, modelling the illumination source. They verify the time-harmonic Maxwell equations,
\begin{equation}\label{eq:maxwell}
\nabla\times\mathbf{E}-\mathrm{i}\omega\mu_\infty\mathbf{H}=0,\quad\nabla\times\mathbf{H}+\mathrm{i}\omega\epsilon_\infty\mathbf{E}=0\quad\mbox{in}\ \ \mathbb{R}^3. 
\end{equation} 
The presence of the inhomogeneous scatterer $(\Sigma;\epsilon,\mu,\sigma)$ interrupts the propagation of the EM waves $\mathbf{E}^{i,\omega}$ and $\mathbf{H}^{i,\omega}$, leading to the so-called wave scattering. We let $\mathbf{E}^{s,\omega}$ and $\mathbf{H}^{s,\omega}$ denote, respectively, the scattered electric and magnetic fields. Define
\begin{equation}\label{eq:total}
\mathbf{E}^\omega=:\mathbf{E}^{i,\omega}+\mathbf{E}^{s,\omega},\quad \mathbf{H}^\omega=:\mathbf{H}^{i,\omega}+\mathbf{H}^{s,\omega},
\end{equation}       
to be the total electric and magnetic fields, respectively. Then the EM scattering is governed by the following Maxwell system
\begin{equation}\label{eq:maxwell2}
\begin{cases}
& \nabla\times\mathbf{E}^\omega(\mathbf{x})-\mathrm{i}\omega\mu(\mathbf{x})\mathbf{H}^\omega(\mathbf{x})=0,\hspace*{2.3cm} \mathbf{x}\in \mathbb{R}^3,\medskip\\
&\nabla\times\mathbf{H}^\omega(\mathbf{x})+\mathrm{i}\omega\epsilon(\mathbf{x})\mathbf{E}^\omega(x)=\sigma(\mathbf{x})\mathbf{E}^\omega(\mathbf{x}),\quad \ \ \mathbf{x}\in \mathbb{R}^3,\medskip\\
&\displaystyle{\lim_{\|{\mathbf{x}}\|\rightarrow+\infty}\big(\mu_\infty^{1/2}\mathbf{H}^{s,\omega}(\mathbf{x})\times \mathbf{x}-\|{\mathbf{x}}\| \epsilon_\infty^{1/2}\mathbf{E}^{s,\omega}(\mathbf{x})\big)=0. }
\end{cases}
\end{equation}
The last limit in \eqref{eq:maxwell2} is known as the Silver-M\"uller radiation condition. The Maxwell system \eqref{eq:maxwell2} is well posed and there exists a unique pair of solutions $(\mathbf{E}^\omega, \mathbf{H}^\omega)\in H_{loc}^2(\mbox{curl}, \mathbb{R}^3)$ (cf. \cite{LRX,Mon,PWW}). Here and also in what follows, we make use of the following Sobolev spaces that for any open set $G\subset\mathbb{R}^3$, 
\[
\begin{split}
H_{loc}(\mathrm{curl},G)& =\{\mathbf{u}\in L^2_{loc}(G;\mathbb{C}^3):\ \nabla\times \mathbf{u}\in L^2_{loc}(G;\mathbb{C}^3)\},\\
H^2_{loc}(\mathrm{curl},G)& =\{\mathbf{u}\in H_{loc}(\mathrm{curl}, G):\ \nabla\times \mathbf{u}\in H_{loc}(\mathrm{curl}, G)\},
\end{split}
\]
Let $\kappa_\omega:=\sqrt{\mu_\infty\epsilon_\infty}\omega$ signify the wavenumber. For the solutions to \eqref{eq:maxwell2}, we have that as $\|\mathbf{x}\|\rightarrow+\infty$ (cf. \cite{CK,Ned}),
\begin{equation*}
\begin{split}
\mathbf{E}^\omega(\mathbf{x})=&\frac{e^{\mathrm{i}\kappa_\omega\|\mathbf{x}\|}}{\|\mathbf{x}\|}\mathbf{E}_\infty^\omega(\hat{\mathbf{x}})+\mathcal{O}\left(\frac{1}{\|\mathbf{x}\|^2}\right),\\
\mathbf{H}^\omega(\mathbf{x})=&\frac{e^{\mathrm{i}\kappa_\omega\|\mathbf{x}\|}}{\|\mathbf{x}\|}\mathbf{H}_\infty^\omega(\hat{\mathbf{x}})+\mathcal{O}\left(\frac{1}{\|\mathbf{x}\|^2}\right),
\end{split}
\end{equation*}
where $\hat{\mathbf{x}}:=\mathbf{x}/\|\mathbf{x}\|\in\mathbb{S}^2$, $\mathbf{x}\in\mathbb{R}^3\backslash\{0\}$. 
$\mathbf{E}_\infty^\omega$ and $\mathbf{H}_\omega^\infty$ are, respectively, referred to as the electric and magnetic far-field patterns, and they satisfy for any $\hat{\mathbf{x}}\in\mathbb{S}^2$,
\begin{equation*}
\mathbf{H}^\omega_\infty(\hat{\mathbf{x}})=\hat{\mathbf{x}}\times \mathbf{E}^\omega_\infty(\hat{\mathbf{x}})\quad \mbox{and}\quad \hat{\mathbf{x}}\cdot \mathbf{E}^\omega_\infty(\hat{\mathbf{x}})=\hat{\mathbf{x}}\cdot \mathbf{H}^\omega_\infty(\hat{\mathbf{x}})=0.
\end{equation*}

In terms of the shadowless lamp setting, $(\Sigma;\epsilon,\mu,\sigma)$ signifies the bodies of the surgeons, whereas $\mathbf{E}^\omega_\infty$ and $\mathbf{H}^\omega_\infty$ account for the production of the shadows. It is emphasised that $\Sigma$ is not necessarily simply connected, and hence we allow the presence of multiple surgeons under the lamp. The major contribution in this paper is to provide a deterministic and constructive way, for a given inhomogeneous scatterer $(\Sigma;\epsilon,\mu,\sigma)$, in deriving a set of entire EM waves $\mathcal{W}$, such that for any $(\mathbf{E}^{i,\omega}, \mathbf{H}^{i,\omega})\in \mathcal{W}$, the corresponding $\mathbf{E}_\infty^\omega$ and $\mathbf{H}_\infty^\omega$ are nearly vanishing. From the practical motivation, we shall consider three scenarios in our study. The first two cases are to coat the inhomogeneous scatterer with certain properly designed layers of conducting mediums. Then the set $\mathcal{W}$ is generated from the Maxwell-Herglotz approximation to the so-called interior PEC or PMC eigenfunctions, and it depends only on the shape of the scatterer. The third case is without such a conducting coating, then the set $\mathcal{W}$ is generated from the Maxwell-Herglotz approximation to certain generalised transmission eigenfunctions, and it depends on both the shape and the content of the scatterer. 

The rest of the paper is organised as follows. In Section 2, we present the major results on nearly non-scattering EM wave fields. Section 3 is devoted to the proofs of the theorems in Section 2.

\section{Nearly non-scattering wave sets}

We first present some preliminary results, including the interior eigenvalue problems and the approximation by Maxwell-Herglotz fields. Let $\Omega$ be a bounded open set in $\mathbb{R}^3$ with a Lipschitz boundary $\partial\Omega$ and a connected complement $\Omega^c$. We begin with the following PEC eigenvalue problem,
\begin{equation}\label{eq:peceigen}
\begin{cases}
& \nabla\times\mathbf{E}^\omega(\mathbf{x})-\mathrm{i}\omega\mu_\infty\mathbf{H}^\omega(\mathbf{x})=0,\quad \ \ \mathbf{x}\in \Omega,\medskip\\
&\nabla\times\mathbf{H}^\omega(\mathbf{x})+\mathrm{i}\omega\epsilon_\infty\mathbf{E}^\omega(\mathbf{x})=0,\quad \ \, \ \mathbf{x}\in \Omega,\medskip\\
&\displaystyle{\bm{\nu}\times \mathbf{E}^\omega(\mathbf{x})=0,\quad \mathbf{x}\in\partial\Omega.  }
\end{cases}
\end{equation}
If there exists a pair of nontrivial solutions $(\mathbf{E}^\omega,\mathbf{H}^\omega)$ to \eqref{eq:peceigen}, then $\omega\in\mathbb{R}_+$ is called a {\it PEC eigenvalue} associated with $(\Omega;\epsilon_\infty,\mu_\infty)$, and $(\mathbf{E}^\omega,\mathbf{H}^\omega)$ is referred to as the corresponding pair of eigenfunctions. Indeed, all the eigenfunctions $\mathbf{E}^\omega$ (resp. $\mathbf{H}^\omega$) associated with the eigenvalue $\omega$ form a vector space, and we shall denote by $\mathbf{X}_\omega$ in what follows. Introducing the Sobolev space  
\begin{equation*}
\mathbf{X}=\{\mathbf{u}\in H(\mathrm{curl}, \Omega);\ \bm{\nu}\times\mathbf{u}=0\ \mbox{on}\ \ \partial\Omega\ \ \mbox{and}\ \ \nabla\cdot\mathbf{u}=0\ \mbox{in}\ \ \Omega \},
\end{equation*}
then we have the following result (cf. \cite{Mon}).
\begin{thm}\label{thm:peceigen}
There exists an infinite discrete set of eigenvalues $\omega_j\in\mathbb{R}_+,\ j=1,2,3,\cdots$ to \eqref{eq:peceigen} and the corresponding eigenfunctions $\mathbf{E}^{\omega_j}\in \mathbf{X}$ such that
\begin{enumerate}
\item The Maxwell system \eqref{eq:peceigen} holds for each pair $(\mathbf{E}^{\omega_j}, \mathbf{H}^{\omega_j})$ with $\mathbf{H}^{\omega_j}:=(\mathrm{i}\omega_j\mu_\infty)^{-1}$ $\nabla\times\mathbf{E}^{\omega_j}$;\smallskip

\item $\omega_1\leq\omega_2\leq\omega_3\leq\cdots$ and $\lim_{j\rightarrow+\infty}\omega_j=+\infty$;\smallskip

\item $\{\mathbf{E}^{\omega_j}\}^\infty_{j=1}$ are orthonormal with respect to the $L^2$ inner product;\smallskip

\item $\mathbf{X}=span\{\mathbf{E}^{\omega_1},\mathbf{E}^{\omega_2},\cdots\}$;\smallskip

\item The eigen-space $\mathbf{X}_{\omega_j}$ is finite dimensional, $j=1,2,\ldots$. 
\end{enumerate}
\end{thm}

Next, we consider the approximation by Maxwell-Herglotz wave fields. 
\begin{defn}
A Maxwell-Herglotz pair is a pair of vector fields of the form
\begin{equation}\label{eq:mh}
\mathbf{E}_{\mathbf{a}}^{\kappa_\omega}(\mathbf{x})=\epsilon_\infty^{-1/2}\int_{\mathbb{S}^2}e^{\mathrm{i}\kappa_\omega\mathbf{x}\cdot \mathbf{d}}\mathbf{a}(\mathbf{d})\,ds(\mathbf{d}),\ \mathbf{H}^{\kappa_\omega}_{\mathbf{a}}(\mathbf{x})=\frac{1}{\mathrm{i}\omega\mu_\infty}\nabla\times\mathbf{E}_{\mathbf{a}}^{\kappa_\omega}(\mathbf{x}),\ \mathbf{x}\in\mathbb{R}^3,
\end{equation}
with 
\[
\mathbf{a}\in TL^2(\mathbb{S}^2):=\{\mathbf{a}\in L^2(\mathbb{S}^2;\mathbb{C}^3); \ \bm{\nu}\cdot\mathbf{a}=0\},
\]
where $\bm{\nu}$ signifies the exterior unit normal vector to $\mathbb{S}^2$. 
\end{defn}

It is straightforward to verify that $(\mathbf{E}_{\mathbf{a}}^{\kappa_\omega},\mathbf{H}_{\mathbf{a}}^{\kappa_\omega})$ is a pair of entire solutions to the Maxwell equations \eqref{eq:maxwell}. There holds,
\begin{thm}[Theorems 2 and 4 in \cite{Wec}]\label{thm:main1}
Suppose that $\Omega$ is a domain of class $C^2$. Then the Maxwell-Herglotz fields are dense in the space of all solutions $(\mathbf{E},\mathbf{H})$ to \eqref{eq:maxwell} which belong to $H^1(\Omega)\times H^1(\Omega)$ in the topology induced by the following norm 
\[
\interleave(\mathbf{E},\mathbf{H})\interleave_1:=\|\mathbf{E}\|_{H^1(\Omega;\mathbb{C}^3)}+\|\mathbf{H}\|_{H^1(\Omega;\mathbb{C}^3)}. 
\]

Suppose that $\Omega$ is a domain of class $C^{0,1}$ and $\omega$ is not a PEC eigenvalue associated with $(\Omega;\epsilon_\infty,\mu_\infty)$. Then the Maxwell-Herglotz fields are dense in the space of all solutions $(\mathbf{E},\mathbf{H})$ to \eqref{eq:maxwell} which belong to $H(\mathrm{curl},\Omega)\times H(\mathrm{curl},\Omega)$ in the topology induced by the following norm 
\[
\interleave (\mathbf{E},\mathbf{H}) \interleave_2:=\|\mathbf{E}\|_{H(\mathrm{curl},\Omega)}+\|\mathbf{H}\|_{H(\mathrm{curl}, \Omega)}. 
\]

\end{thm}

Starting from now on, we assume that the domain $\Omega$ introduced in \eqref{eq:peceigen} has a $C^2$-smooth boundary $\partial\Omega$. We define
\begin{equation*}
\mathbf{Y}=\{(\mathbf{u},\mathbf{v})\in L^2(\Omega;\mathbb{C}^3)^2;\ \mathbf{u}\in\mathbf{X},\ \mathbf{v}=(\mathrm{i}\omega\mu_\infty)^{-1}\nabla\times\mathbf{u}\}
\end{equation*} 
to be the vector space consisting of all of the eigen-pairs to \eqref{eq:peceigen}. By the standard regularity estimate (cf. \cite{Mon,Ned}), we know that for any $(\mathbf{u},\mathbf{v})\in \mathbf{Y}$, $(\mathbf{u},\mathbf{v})\in H^1(\Omega)\times H^1(\Omega)$. By Theorem~\ref{thm:main1}, for any $\varepsilon>0$, we let $\mathcal{H}_\varepsilon(\mathbf{Y})$ denote an $\varepsilon$-net of the set $\mathbf{Y}$ in the space $H^1(\Omega)\times H^1(\Omega)$ in the following; that is, for any $(\mathbf{u},\mathbf{v})\in\mathbf{Y}$ , there exists $(\widetilde{\mathbf{u}},\widetilde{\mathbf{v}}) \in \mathcal{H}_\varepsilon(\mathbf{Y})$, being a Maxwell-Herglotz pair of the form \eqref{eq:mh}, such that 
\begin{equation}\label{eq:net1}
\interleave (\widetilde{\mathbf{u}},\widetilde{\mathbf{v}})-(\mathbf{u},\mathbf{v}) \interleave_1\leq \varepsilon. 
\end{equation}

We are in a position to present our first major result on the nearly non-scattering wave set. 

\begin{thm}\label{thm:nonscattering1}
Let $\Sigma$ be a bounded Lipschitz domain and $\Omega$ be a bounded $C^2$ domain such that $\Sigma\Subset\Omega$ and $\Omega^c$ is connected. Consider an EM medium distribution as follows,
\begin{equation}\label{eq:emp1}
\mathbb{R}^3;\epsilon,\mu,\sigma=\begin{cases}
\Sigma;\ \ \epsilon_b,\mu_b,\sigma_b,\\
\Omega\backslash\overline{\Sigma};\ \ \tau^{-1}\epsilon_c, \tau \mu_c, \tau^{-1}\sigma_c,\\
\mathbb{R}^3\backslash\overline{\Omega};\ \ \epsilon_\infty\cdot\mathbf{I}_{3\times 3}, \mu_\infty\cdot\mathbf{I}_{3\times 3}, 0\cdot\mathbf{I}_{3\times 3},
\end{cases}
\end{equation}
where $\epsilon_b, \mu_b$ and $\sigma_b$ satisfy \eqref{eq:elliptic} and \eqref{eq:elliptic2}, respectively; and $\epsilon_c, \mu_c$ and $\sigma_c$ are uniformly elliptic tensors with constant $\alpha_0\in\mathbb{R}_+$, and moreover it is assumed that $\mu_c=\beta_0\cdot\mathbf{I}_{3\times 3}$ with $\beta_0\in\mathbb{R}_+$. Consider the electromagnetic scattering problem \eqref{eq:maxwell2} associated with the EM medium $(\mathbb{R}^3;\epsilon,\mu,\sigma)$ described above, and a pair of incident fields $(\mathbf{E}^{i,\omega},\mathbf{H}^{i,\omega})=(\mathbf{E}_{\mathbf{a}}^{\kappa_\omega},\mathbf{H}_{\mathbf{a}}^{\kappa_\omega})\in\mathcal{H}_\varepsilon(\mathbf{Y})$. Then for sufficiently small $\tau$ and $\varepsilon$, we have
\begin{equation}\label{eq:est1}
\|\mathbf{E}_\infty^\omega\|_{L^2(\mathbb{S}^2; \mathbb{C}^3)}\leq C\Big( \tau^{1/2}\big\|\mathbf{E}_{\mathbf{a}}^{\kappa_\omega} \big\|_{H(\mathrm{curl}, \Omega)}+\tau^{1/2}\big\|\mathbf{H}_{\mathbf{a}}^{\kappa_\omega} \big\|_{H(\mathrm{curl}, \Omega)}+\varepsilon \Big),
\end{equation}
where $C$ depends only on $\omega, \alpha_0, \epsilon_\infty, \mu_\infty$ and $\Omega, \Sigma$, but independent of $\epsilon_b, \mu_b$ and $\sigma_b$. 

\end{thm}

\begin{rem}\label{rem:2.1}
In terms of our earlier discussion in Section 1, $\mathcal{H}_\varepsilon(\mathbf{Y})$ is a nearly non-scattering wave set for the EM medium $(\mathbb{R}^3; \epsilon,\mu,\sigma)$ in \eqref{eq:est1}. Every PEC eigenvalue $\omega$ to \eqref{eq:peceigen} is a nearly non-scattering ``energy" in the sense that there exist certain incident fields of the Maxwell-Herglotz form \eqref{eq:mh} which generate nearly vanishing scattered wave fields. In the shadowless lamp setting, $(\Sigma; \epsilon_b, \mu_b, \sigma_b)$ signifies the surgeon(s), whereas $(\Omega\backslash\overline{\Sigma}; \tau^{-1}\epsilon_c, \tau\mu_c, \tau^{-1}\sigma_c)$ signifies suitable coat(s) made of conducting mediums. According to Theorem~\ref{thm:nonscattering1}, if the sources of the lamp are chosen from $\mathcal{H}_\varepsilon(\mathbf{Y})$, then there will be nearly no shadows shall be generated. 
\end{rem}

In a similar fashion, we define
\begin{equation*}
\mathbf{Z}=\{(\mathbf{u},\mathbf{v})\in L^2(\Omega;\mathbb{C}^3)^2;\ \mathbf{v}\in\mathbf{X},\ \mathbf{u}=-(\mathrm{i}\omega\epsilon_\infty)^{-1}\nabla\times\mathbf{v}\}.
\end{equation*} 
It is readily seen that any $(\mathbf{E}^\omega,\mathbf{H}^\omega)\in\mathbf{Z}$ satisfies the Maxwell system \eqref{eq:peceigen}, but with the homogeneous PEC boundary condition on $\partial\Omega$ replaced by the so-called PMC boundary condition,
\begin{equation*}
\bm{\nu}\times\mathbf{H}^\omega=0\quad\mbox{on}\ \ \partial\Omega. 
\end{equation*}
Similar to Theorem~\ref{thm:peceigen}, we know that $\mathbf{Z}$ contains infinitely many pairs of PMC eigenfunctions and we let $\mathcal{H}_\varepsilon(\mathbf{Z})$ denote an $\varepsilon$-net of the space $\mathbf{Z}$ in the sense of \eqref{eq:net1}. Then we have
\begin{thm}\label{thm:nonscattering2}
Let $\Sigma$ be a bounded Lipschitz domain and $\Omega$ be a bounded $C^2$ domain such that $\Sigma\Subset\Omega$ and $\Omega^c$ are connected. Consider an EM medium distribution as follows,
\begin{equation}\label{eq:emp12}
\mathbb{R}^3;\epsilon,\mu,\sigma=\begin{cases}
\Sigma;\ \ \epsilon_b,\mu_b,\sigma_b,\\
\Omega\backslash\overline{\Sigma};\ \ \tau\epsilon_c, \tau^{-1} \mu_c, \tau \sigma_c,\\
\mathbb{R}^3\backslash\overline{\Omega};\ \ \epsilon_\infty\cdot\mathbf{I}_{3\times 3}, \mu_\infty\cdot\mathbf{I}_{3\times 3}, 0\cdot\mathbf{I}_{3\times 3},
\end{cases}
\end{equation}
where $\mu_b$ is uniformly elliptic; and $\epsilon_b, \sigma_b$ and $\epsilon_c, \mu_c, \sigma_c$ are uniformly elliptic tensors with constant $\alpha_0\in\mathbb{R}_+$, and moreover it is assumed that $\epsilon_c=\beta_0\cdot\mathbf{I}_{3\times 3}$ and $\sigma_c=\eta_0\cdot\mathbf{I}_{3\times 3}$ with $\beta_0, \eta_0\in\mathbb{R}_+$. Consider the electromagnetic scattering problem \eqref{eq:maxwell2} associated with the EM medium $(\mathbb{R}^3;\epsilon,\mu,\sigma)$ described above, and a pair of incident fields $(\mathbf{E}^{i,\omega},\mathbf{H}^{i,\omega})=(\mathbf{E}_{\mathbf{a}}^{\kappa_\omega},\mathbf{H}_{\mathbf{a}}^{\kappa_\omega})\in\mathcal{H}_\varepsilon(\mathbf{Z})$. Then for sufficiently small $\tau$ and $\varepsilon$, we have
\begin{equation}\label{eq:est12}
\|\mathbf{E}_\infty^\omega\|_{L^2(\mathbb{S}^2; \mathbb{C}^3)}\leq C\Big( \tau^{1/2}\big\|\mathbf{E}_{\mathbf{a}}^{\kappa_\omega} \big\|_{H(\mathrm{curl},\Omega)}+ \tau^{1/2}\big\|\mathbf{H}_{\mathbf{a}}^{\kappa_\omega} \big\|_{H(\mathrm{curl}, \Omega)}+\varepsilon \Big),
\end{equation}
where $C$ depends on $\omega, \alpha_0, \epsilon_\infty, \mu_\infty, \epsilon_b, \sigma_b$ and $\Omega, \Sigma$, but independent of $\mu_b$. 

\end{thm}

\begin{rem}
The major difference between Theorems~\ref{thm:nonscattering1} and $\ref{thm:nonscattering2}$ is that in \eqref{eq:est1}, the estimate is independent of the EM content inside $\Sigma$, namely $(\Sigma; \epsilon_b,\mu_b,\sigma_b)$; whereas in \eqref{eq:est12}, the estimate is dependent on $(\Sigma; \epsilon_b, \sigma_b)$ (but independent of $\mu_b$). We believe this is mainly due to the argument that we shall implement for their proofs, and the estimate in \eqref{eq:est12} should also be independent of $(\Sigma; \epsilon_b,\mu_b,\sigma_b)$. 
\end{rem}

In Theorems~\ref{thm:nonscattering1} and \ref{thm:nonscattering2}, the conducting layer in $\Omega\backslash\overline{\Sigma}$ plays a critical role in our design. In what follows, we consider the nearly non-scattering wave fields in the case without such a conducting layer. To that end, we first introduce the following interior transmission problem.   
\begin{equation}\label{eq:gitp}
\begin{cases}
&\nabla\times\mathbf{E}^{t,\omega}(\mathbf{x})-\mathrm{i}\omega\mu_b\mathbf{H}^{t,\omega}(\mathbf{x})=0,\qquad\mathbf{x}\in\Sigma,\medskip\\
&\nabla\times\mathbf{H}^{t,\omega}(\mathbf{x})+\mathrm{i}\omega\epsilon_b\mathbf{E}^{t,\omega}(\mathbf{x})=0,\qquad\, \mathbf{x}\in\Sigma,\medskip\\
&\nabla\times\mathbf{U}^\omega(\mathbf{x})-\mathrm{i}\omega\mu_\infty\mathbf{V}^\omega(\mathbf{x})=0,\qquad\ \mathbf{x}\in\Sigma,\medskip\\
&\nabla\times\mathbf{V}^\omega(\mathbf{x})-\mathrm{i}\omega\epsilon_\infty\mathbf{U}^\omega(\mathbf{x})=0,\qquad\ \, \mathbf{x}\in\Sigma,\medskip\\
&\bm{\nu}\times\mathbf{E}^{t,\omega}(\mathbf{x})=\bm{\nu}\times\mathbf{U}^\omega(\mathbf{x}),\ \ \bm{\nu}\times\mathbf{H}^{t,\omega}(\mathbf{x})=\bm{\nu}\times\mathbf{V}^\omega(\mathbf{x}),\ \ \mathbf{x}\in\partial\Sigma,
\end{cases}
\end{equation}
where $\mu_b$ and $\epsilon_b$ are uniformly elliptic tensors in $\Sigma$. If there exists nontrivial solutions $(\mathbf{E}^{t,\omega},\mathbf{H}^{t,\omega})\in H(\text{curl},\Sigma)^2$ and $(\mathbf{U}^\omega, \mathbf{V}^{\omega})\in H(\text{curl},\Omega)^2$ to \eqref{eq:gitp} for a certain $\omega\in\mathbb{R}_+$, then $\omega$ is called an {\it interior transmission eigenvalue} associated with $(\Sigma;\epsilon_b,\mu_b,\epsilon_\infty,\mu_\infty)$, and $(\mathbf{E}^{t,\omega},\mathbf{H}^{t,\omega})$, $(\mathbf{U}^\omega,\mathbf{V}^\omega)$ are called the corresponding {\it interior transmission eigenfunctions}. Some brief remarks about the interior transmission eigenvalue problem \eqref{eq:gitp} are in order. We refer to \cite{CakCol} for a comprehensive account on the origin of the interior transmission eigenvalue problems. The existence of an infinite discrete set of interior transmission eigenvalues with $+\infty$ as the only accumulation point for \eqref{eq:gitp} was established in \cite{CH} under certain conditions on the scattering medium $(\Sigma;\epsilon_b,\mu_b)$. Establishing the existence of infinitely many interior transmission eigenvalues for \eqref{eq:gitp} for generic EM mediums $(\Sigma;\epsilon_b,\mu_b)$ is beyond the main aim of the present article. We shall assume the existence of interior transmission eigenvalues and eigenfunctions for \eqref{eq:gitp} without imposing further conditions on $(\Sigma;\epsilon_b,\mu_b)$. The notion of non-scattering energy associated with the interior transmission eigenvalue problems was introduced \cite{BPS}, which is closely related to our study in the sequel. In \cite{BPS}, the authors consider the quantum scattering governed by the Schr\"odiner equation. It is defined for a certain potential that the interior transmission eigenvalue is a non-scattering energy if the corresponding eigenfunction is a Herglotz wave function. In the context of the Maxwell system \eqref{eq:gitp}, an interior transmission eigenvalue $\omega$ is called a non-scattering ``energy" if the interior transmission eigenfunctions $(\mathbf{U}^\omega,\mathbf{V}^\omega)$ happen to be a pair of Maxwell-Herglotz fields of the form \eqref{eq:mh}. It can be easily shown that if $\omega$ is a non-scattering ``energy", and one uses the corresponding pair of eigenfunctions $(\mathbf{U}^\omega,\mathbf{V}^\omega)$ as the incident fields to impinge on $(\Sigma;\epsilon_b,\mu_b)$, then far-field pattern generated will be identically vanishing. However, as in \cite{BPS}, it is highly suspicious whether there exists non-scattering $\omega$ for a generic $(\Sigma; \epsilon_b,\mu_b)$ unless it is radially symmetric. Next, we shall show that as soon as the set $\Sigma$ satisfies a certain generic condition, then every interior transmission eigenvalue $\omega$ is a nearly non-scattering ``energy"; see also our discussion in Remark~\ref{rem:2.1} about the existence of infinitely many nearly non-scattering ``energy" for the Maxwell system in a different scenario. 

In our subsequent study concerning Theorem~\ref{thm:nonscattering3}, we assume that $\Sigma$ is a bounded Lipschitz domain if $\omega$ is not a PEC eigenvalue associated with $(\Sigma;\epsilon_\infty, \mu_\infty)$; otherwise we assume that $\Sigma$ is a $C^2$ domain. This regularity assumption shall be mainly needed for the Maxwell-Herglotz approximation as sated in Theorem~\ref{thm:main1}. In the sequel, we shall need to make use of the following Sobolev spaces,
\[
\begin{split}
TH^{-1/2}(\Sigma)=\{\mathbf{u}\in H^{-1/2}(\Sigma;\mathbb{C}^3);\ \langle\bm{\nu},\mathbf{u}\rangle=0\},\\
TH_{\text{Div}}^{-1/2}(\Sigma)=\{\mathbf{u}\in TH^{-1/2}(\Sigma);\ \mathrm{Div}(\mathbf{u})\in H^{-1/2}(\Sigma)\},\\
TH_{\text{Curl}}^{-1/2}(\Sigma)=\{\mathbf{u}\in TH^{-1/2}(\Sigma);\ \mathrm{Curl}(\mathbf{u})\in H^{-1/2}(\Sigma)\},
\end{split}
\] 
where $\mathrm{Div}:=\mathrm{div}_{\partial\Sigma}$ and $\mathrm{Curl}:=\mathrm{curl}_{\partial\Sigma}$, respectively, signify the surface divergence and curl on $\partial\Sigma$. It is known that $TH_{\text{Div}}^{-1/2}(\Sigma)$ is the tangential trace space of $H(\mathrm{curl}, \Sigma)$, and the dual space of $TH_{\text{Div}}^{-1/2}(\Sigma)$ is $TH_{\text{Curl}}^{-1/2}(\Sigma)$; see \cite{Alo,BCS,Ces,Mon}.

We let $\mathcal{T}=\mathcal{T}(\Sigma;\epsilon_b,\mu_b)$ denote the set of interior transmission eigenvalues associated with $(\Sigma;\epsilon_b,\mu_b)$ defined in \eqref{eq:gitp}, and $\mathbf{T}_\omega=\mathbf{T}_\omega(\Sigma;\epsilon_b,\mu_b)$ denote the eigen-space consisting of the pairs $(\mathbf{U}^\omega,\mathbf{V}^\omega)$ corresponding to $\omega\in\mathcal{T}$. We also set
\begin{equation*}
\mathbf{T}=\bigcup_{\omega\in\mathcal{T}}\mathbf{T}_\omega. 
\end{equation*}

Let $\omega\in\mathcal{T}$. We further assume that $\omega$ is not a PEC eigenvalue for $(\Sigma;\epsilon_b,\mu_b)$ in the sense that the following Maxwell system
\begin{equation}\label{eq:impedanceinside}
\begin{cases}
& \nabla\times\mathbf{E}^\omega(\mathbf{x})-\mathrm{i}\omega\mu_b\mathbf{H}^\omega(\mathbf{x})=0,\quad \ \ \ \mathbf{x}\in \Sigma,\medskip\\
&\nabla\times\mathbf{H}^\omega(\mathbf{x})+\mathrm{i}\omega\epsilon_b\mathbf{E}^\omega(\mathbf{x})=0,\quad \ \ \ \mathbf{x}\in \Sigma,\medskip\\
&\displaystyle{\bm{\nu}\times \mathbf{E}^\omega(\mathbf{x})=0,\quad \mathbf{x}\in\partial\Sigma.  }
\end{cases}
\end{equation}
admits only trivial solutions. Define the interior boundary {\it impedance map} as
\begin{equation*}
\Lambda_{\Sigma,\omega}^i(\bm{\psi})=\bm{\nu}\times\mathbf{H}^\omega\big|_{\partial\Sigma}:\ TH_{\text{Div}}^{-1/2}(\partial\Sigma)\rightarrow TH_{\text{Div}}^{-1/2}(\partial\Sigma),
\end{equation*}
where $\mathbf{H}^\omega\in H(\text{curl}, \Sigma)$ is the solution to \eqref{eq:impedanceinside} with the homogeneous boundary condition replaced by
\[
{\bm{\nu}\times \mathbf{E}^\omega(\mathbf{x})=\bm{\psi}(\mathbf{x})\in TH_{\text{Div}}^{-1/2}(\partial\Sigma),\quad \mathbf{x}\in\partial\Sigma.  }
\]
Clearly, since $\omega$ is not a PEC for $(\Sigma; \epsilon_b, \mu_b)$, the Maxwell system \eqref{eq:impedanceinside} is well-posed and hence $\Lambda_{\Sigma,\omega}^i$ is well-defined. It is also assumed that $\omega$ is not a PMC eigenvalue for $(\Sigma;\epsilon_b,\mu_b)$ in the sense that the Maxwell system \eqref{eq:impedanceinside} with the homogeneous boundary condition replaced by $\bm{\nu}\times\mathbf{H}^\omega=0$ on $\partial\Sigma$ admits only trivial solutions. Hence, we know that $\Lambda_{\Sigma,\omega}^i$ is invertible. It is remarked that the PEC (resp. PMC) eigenvalues for $(\Sigma;\epsilon_b,\mu_b)$ form an infinite discrete set possessing similar properties to those stated in Theorem~\ref{thm:peceigen} (cf. \cite{Mon}). We also consider the following exterior Maxwell system 
\begin{equation}\label{eq:impedanceoutside}
\begin{cases}
& \nabla\times\mathbf{E}^\omega(\mathbf{x})-\mathrm{i}\omega\mu_\infty\mathbf{H}^\omega(\mathbf{x})=0,\quad \ \ \ \mathbf{x}\in \mathbb{R}^3\backslash\overline{\Sigma},\medskip\\
&\nabla\times\mathbf{H}^\omega(\mathbf{x})+\mathrm{i}\omega\epsilon_\infty\mathbf{E}^\omega(\mathbf{x})=0,\quad \ \ \ \mathbf{x}\in \mathbb{R}^3\backslash\overline{\Sigma},\medskip\\
&\displaystyle{\lim_{\|{\mathbf{x}}\|\rightarrow+\infty}\big(\mu_\infty^{1/2}\mathbf{H}^{\omega}(\mathbf{x})\times \mathbf{x}-\|{\mathbf{x}}\| \epsilon_\infty^{1/2}\mathbf{E}^{\omega}(\mathbf{x})\big)=0. }\medskip\\
&\displaystyle{\bm{\nu}\times \mathbf{E}^\omega(\mathbf{x})=\bm{\psi}(\mathbf{x})\in TH_{\text{Div}}^{-1/2}(\partial\Sigma),\quad \mathbf{x}\in\partial\Sigma.  }
\end{cases}
\end{equation}
Define the exterior boundary impedance map as 
\begin{equation}\label{eq:oepm1}
\Lambda_{\Sigma,\omega}^o(\bm{\psi})=\bm{\nu}\times\mathbf{H}^\omega\big|_{\partial\Sigma}:\ TH_{\text{Div}}^{-1/2}(\partial\Sigma)\rightarrow TH_{\text{Div}}^{-1/2}(\partial\Sigma),
\end{equation}
where $\mathbf{H}^\omega\in H_{doc}(\text{curl}, \mathbb{R}^3\backslash\overline{\Sigma})$ is the solution to \eqref{eq:impedanceoutside}. Since the exterior scattering problem \eqref{eq:impedanceoutside} is well-posed, one clearly has that $\Lambda_{\Sigma,\omega}^o$ is well-defined and moreover, it is invertible. Next, we say that $\Sigma$ satisfies the non-transparency condition with respect to $\omega$ if the following condition is satisfied:
\begin{equation}\label{eq:nontransparency1}
\big\|\Lambda_{\Sigma,\omega}^i\circ\big(\Lambda_{\Sigma,\omega}^o\big)^{-1}\big\|_{\mathcal{L}(TH_{\text{Div}}^{-1/2}(\partial\Sigma),TH_{\text{Div}}^{-1/2}(\partial\Sigma))}\neq 1;
\end{equation}
A remark concerning the non-transparency condition is that unless $\bm{\psi}\equiv 0$, one has that 
\begin{equation}\label{eq:rt1}
\Lambda_{\Sigma,\omega}^i\circ\big(\Lambda_{\Sigma,\omega}^o\big)^{-1}(\bm{\psi})\neq\bm{\psi}\quad\mbox{for any}\ \ \bm{\psi}\in TH_{\text{Div}}^{-1/2}(\partial\Sigma). 
\end{equation}
Indeed, if the equality holds in \eqref{eq:rt1} for some $\bm{\psi}\in TH_{\text{Div}}^{-1/2}(\partial\Sigma)$, then by defining $(\mathbf{E}^\omega,\mathbf{H}^\omega)\in H_{loc}(\text{curl},\mathbb{R}^3)^2$ to be the solution to \eqref{eq:impedanceinside} in $\Sigma$, and to be the solution to \eqref{eq:impedanceoutside} in $\mathbb{R}^3\backslash\overline{\Sigma}$, one has a pair of entire solutions to the Maxwell system which satisfies the Silver-M\"uller radiation condition. Hence, the entire solutions defined above must be identically zero which readily gives that $\bm{\psi}\equiv 0$. Therefore, it is justifiable to claim that the non-transparency condition introduced above is a generic condition for the domain $\Sigma$ and its interior transmission eigenvalue $\omega$.  

We have

\begin{thm}\label{thm:nonscattering3}
Let $\Sigma$ be a bounded domain with $\Sigma^c$ connected. Consider an EM medium distribution as follows,
\begin{equation*}
\mathbb{R}^3;\epsilon,\mu=\begin{cases}
\Sigma;\ \ \epsilon_b,\ \mu_b,\\
\mathbb{R}^3\backslash\overline{\Sigma};\ \ \epsilon_\infty\cdot\mathbf{I}_{3\times 3},\ \mu_\infty\cdot\mathbf{I}_{3\times 3},
\end{cases}
\end{equation*}
where $\epsilon_b, \mu_b$ are uniformly elliptic tensors. Consider the electromagnetic scattering problem \eqref{eq:maxwell2} associated with the EM medium $(\mathbb{R}^3;\epsilon,\mu)$ described above. Let $\omega\in\mathcal{T}(\Sigma;\epsilon_b,\mu_b)$ and assume that $\omega$ is not a PEC/PMC eigenvalue to $(\Sigma;\epsilon_b,\mu_b)$, and that $\Sigma$ satisfies the non-transparency condition with respect to $\omega$, namely \eqref{eq:nontransparency1}. Let the pair of incident fields be given as
\begin{equation*}
(\mathbf{E}^{i,\omega},\mathbf{H}^{i,\omega})=(\mathbf{E}_{\mathbf{a}}^{\kappa_\omega},\mathbf{H}_{\mathbf{a}}^{\kappa_\omega})\in\mathcal{H}_\varepsilon(\mathbf{T}_\omega),
\end{equation*} 
where by Theorem~\ref{thm:main1}, $\mathcal{H}_\varepsilon(\mathbf{T}_\omega)$ denotes an $\varepsilon$-net of the space $\mathbf{T}_\omega$ in the sense of \eqref{eq:net1}, but with the norm $|\|\cdot\||_1$ replaced by $|\|\cdot\||_2$. For sufficiently small $\varepsilon\in\mathbb{R}_+$, there holds
\begin{equation}\label{eq:est123}
\|\mathbf{E}_\infty^\omega\|_{L^2(\mathbb{S}^2; \mathbb{C}^3)}\leq C\varepsilon,
\end{equation}
where $C$ depends only on $\omega, \epsilon_\infty, \mu_\infty$ and $\Sigma$.

\end{thm}

\section{Proofs of the major theorems}

In this section, we present the proofs for the major theorems in Section 2. 

\subsection{Proof of Theorem~\ref{thm:nonscattering1}}

In order prove Theorem~\ref{thm:nonscattering1}, we first derive two lemmas. In what follows, we shall make use of the exterior impedance map defined in \eqref{eq:oepm1} associated with $(\mathbb{R}^3\backslash\overline{\Omega}; \epsilon_\infty,\mu_\infty)$, which is denoted by $\Lambda_{\Omega,\omega}^o: TH_{\mathrm{Div}}^{-1/2}(\partial\Omega)\rightarrow TH_{\mathrm{Div}}^{-1/2}(\partial\Omega)$. 

\begin{lem}\label{lem:l11}
There holds,
\begin{equation}\label{eq:l11}
\begin{split}
&\Big\|\mathbf{E}^\omega\Big\|^2_{L^2(\Omega\backslash\overline{\Sigma};\mathbb{C}^3)}\\
\leq & C\tau\|\mathbf{E}^{\kappa_\omega}_a\|_{H(\text{\emph{curl}},\Omega)}^2+C\tau\|\mathbf{E}_{\mathbf{a}}^{\kappa_\omega}\|_{H(\mathrm{curl}, \Omega)}\Big\|\Lambda_{\Omega,\omega}^o(\bm{\nu}\times \mathbf{E}^{s,\omega}\big|_{\partial\Omega})\Big\|_{TH_{\mathrm{Div}}^{-1/2}(\partial\Omega)}\\
&+C\tau\|\mathbf{H}_{\mathbf{a}}^{\kappa_\omega}\|_{H(\mathrm{curl},\Omega)}\Big\|\bm{\nu}\times \mathbf{E}^{s,\omega}\Big\|_{TH_{\mathrm{Div}}^{-1/2}(\partial\Omega)}\\
&+C\tau\Big\|\bm{\nu}\times \mathbf{E}^{s,\omega}\Big\|_{TH_{\mathrm{Div}}^{-1/2}(\partial\Omega)}\Big\|\Lambda_{\Omega,\omega}^o(\bm{\nu}\times \mathbf{E}^{s,\omega}\big|_{\partial\Omega})\Big\|_{TH_{\mathrm{Div}}^{-1/2}(\partial\Omega)},
\end{split}
\end{equation}
where $C$ is a positive constant depending only on $\Omega$ and $\alpha_0, \omega, \epsilon_\infty, \mu_\infty$. 
\end{lem}
\begin{proof}
It is first noted that $(\mathbf{E}^\omega,\mathbf{H}^\omega)\in H(\text{curl},\Omega)^2$ satisfies the following Maxwell system 
\begin{equation}\label{eq:pr11}
\begin{split}
&\nabla\times\mathbf{E}^\omega-\mathrm{i}\omega\mu\mathbf{H}^\omega=0,\\
&\nabla\times\mathbf{H}^\omega+\mathrm{i}\omega\epsilon\mathbf{E}^\omega=\sigma\mathbf{E}^\omega,
\end{split}\quad\text{in}\ \ \Omega,
\end{equation}
where $(\Omega;\epsilon,\mu,\sigma)$ is given in \eqref{eq:emp1}. By \eqref{eq:pr11} and using integration by parts, one can calculate as follows,
\begin{equation}\label{eq:pr12}
\begin{split}
&\int_{\Omega}\sigma\mathbf{E}^\omega\cdot\overline{\mathbf{E}^\omega}\, dV=\int_\Omega\left(\nabla\times\mathbf{H}^\omega+\mathrm{i}\omega\epsilon\mathbf{E}^\omega\right)\cdot\overline{\mathbf{E}^\omega}\, dV\\
=&\int_{\Omega}\left(\nabla\times\overline{\mathbf{E}^\omega}\right)\cdot\mathbf{H}^\omega\, dV-\int_{\partial\Omega}\left(\bm{\nu}\times\overline{\mathbf{E}^\omega}\right)\cdot \mathbf{H}^\omega\, ds+\mathrm{i}\omega\int_{\Omega}\epsilon\mathbf{E}^\omega\cdot\overline{\mathbf{E}^\omega}\, dV\\
=&-\mathrm{i}\omega\int_{\Omega}\mu\overline{\mathbf{H}^\omega}\cdot\mathbf{H}^\omega\, dV+\mathrm{i}\omega\int_{\Omega}\epsilon\overline{\mathbf{E}^\omega}\cdot\mathbf{E}^\omega\, dV-\int_{\partial\Omega}(\bm{\nu}\times\overline{\mathbf{E}^\omega})\cdot\mathbf{H}^\omega\, ds. 
\end{split}
\end{equation}
By taking the real parts of both sides of \eqref{eq:pr12}, we have
\begin{equation}\label{eq:pr13}
\int_{\Omega}\sigma\mathbf{E}^\omega\cdot\overline{\mathbf{E}^\omega}\, dV=-\Re\int_{\partial\Omega}(\bm{\nu}\times\overline{\mathbf{E}^\omega})\cdot\mathbf{H}^\omega\, ds.
\end{equation}
Using
\[
\mathbf{E}^\omega=\mathbf{E}^{i,\omega}+\mathbf{E}^{s,\omega},\quad \mathbf{H}^\omega=\mathbf{H}^{i,\omega}+\mathbf{H}^{s,\omega},
\]
and
\[
\bm{\nu}\times\mathbf{H}^{s,\omega}\big|_{\partial\Omega}=\Lambda_{\Omega,\omega}^o\big(\bm{\nu}\times\mathbf{E}^{s,\omega}\big|_{\partial\Omega}\big),
\]
one can calculate that 
\begin{equation*}
\begin{split}
&\int_\Omega\big(\bm{\nu}\times\overline{\mathbf{E}^\omega}\big)\cdot\mathbf{H}^\omega\, dV\\
=&\int_{\partial\Omega}\big(\bm{\nu}\times\overline{\mathbf{E}^{i,\omega}}\big)\cdot\mathbf{H}^{i,\omega}\, ds+\int_{\partial\Omega}\big(\bm{\nu}\times\overline{\mathbf{E}^{i,\omega}}\big)\cdot\big(\Lambda_{\Omega,\omega}^o\big(\bm{\nu}\times\mathbf{E}^{s,\omega}\big|_{\partial\Omega}\big)\times\bm{\nu} \big)\, ds\\
&+\int_{\partial\Omega}\big(\bm{\nu}\times\overline{\mathbf{E}^{s,\omega}}\big)\cdot\mathbf{H}^{i,\omega}\, ds
+\int_{\partial\Omega}\big(\bm{\nu}\times\overline{\mathbf{E}^{s,\omega}}\big)\cdot\big(\Lambda_{\Omega,\omega}^o\big(\bm{\nu}\times\mathbf{E}^{s,\omega}\big|_{\partial\Omega}\big)\times\bm{\nu} \big)\, ds.
\end{split}
\end{equation*}
By integration by parts and straightforward calculations, one has
\begin{equation*}
\int_{\partial\Omega}\big(\bm{\nu}\times\overline{\mathbf{E}^{i,\omega}}\big)\cdot\mathbf{H}^{i,\omega}\, ds=-\mathrm{i}\omega\mu_\infty\int_{\Omega}\big|\mathbf{H}^{i,\omega}\big|^2\, dV+\mathrm{i}\omega\epsilon_\infty\int_{\Omega}\big|\mathbf{E}^{i,\omega}\big|^2\, dV.
\end{equation*}
Clearly, we have
\begin{equation}\label{eq:pr16}
\int_{\Omega}\sigma\mathbf{E}^\omega\cdot\overline{\mathbf{E}^\omega}\, dV\geq \alpha_0\tau^{-1}\|\mathbf{E}^\omega\|^2_{L^2(\Omega\backslash\overline{\Sigma}; \mathbb{C}^3)}.
\end{equation}
Finally, by combining \eqref{eq:pr13}--\eqref{eq:pr16} and using the fact that the skew-symmetric bilinear form $\mathcal{B}: TH_{\mathrm{Div}}^{-1/2}(\partial\Omega)\times TH_{\mathrm{Div}}^{-1/2}(\partial\Omega)$ defined by
\[
\mathcal{B}(\mathbf{j},\mathbf{m})=\int_{\partial\Omega}\mathbf{j}\cdot(\mathbf{m}\times\bm{\nu})\, ds,\quad\forall \, \mathbf{j}, \mathbf{m}\in TH_{\mathrm{Div}}^{-1/2}(\partial\Omega),
\]
is a non-degenerate duality product (cf. \cite{CL}), one can easily show \eqref{eq:l11}. 

The proof is complete. 
\end{proof}

\begin{lem}\label{lem:l12}
There holds,
\begin{equation}\label{eq:l121}
\|\bm{\nu}\times\mathbf{E}^\omega\|_{TH_{\mathrm{Div}}^{-1/2}(\partial\Omega)}\leq C\|\mathbf{E}^\omega\|_{L^2(\Omega\backslash\overline{\Sigma}; \mathbb{C}^3)},
\end{equation}
where $C$ depends only on $\Omega$ and $\alpha_0, \omega, \epsilon_\infty, \mu_\infty$. 
\end{lem}

\begin{proof}
We shall make use of the following duality relation,  
\begin{equation}\label{eq:l122}
\|\bm{\nu}\times\mathbf{E}^\omega\|_{TH_{\mathrm{Div}}^{-1/2}(\partial\Omega)}=\sup_{\|\bm{\psi}\|_{TH^{-1/2}_{\mathrm{Curl}}(\partial\Omega)}\leq 1}\left|\int_{\partial\Omega}\Big(\bm{\nu}\times\mathbf{E}^\omega\Big)\cdot\bm{\psi}\, ds \right|
\end{equation}
For any $\bm{\psi}\in TH^{-1/2}_{\mathrm{Curl}}(\partial\Omega)$, we let $\mathbf{F}\in H^2(\mathrm{curl}, \Omega)$
be such that (see Lemma 3.5 in \cite{BLZ})
\begin{enumerate}
\item $\bm{\nu}\times\mathbf{F}=0$\ on\ $\partial \Omega$; \smallskip
\item $\bm{\nu}\times\bm{\nu}\times(\nabla\times\mathbf{F})=\bm{\nu}\times\bm{\nu}\times\bm{\psi}$;\smallskip
\item $\|\mathbf{F}\|_{H^2(\mathrm{curl}, \Omega)}\leq C\|\bm{\psi}\|_{TH^{-1/2}_{\mathrm{Curl}}(\partial\Omega)}$, where $C$ depends only on $\Omega$. \smallskip
\item $\mathbf{F}=0$ in $\Sigma$. 
\end{enumerate}
By virtue of the duality relation \eqref{eq:l122} and using the auxiliary function $\mathbf{F}$, along with the integration by parts, we have
\begin{equation}\label{eq:l23}
\begin{split}
&\int_{\partial\Omega}\Big(\bm{\nu}\times\mathbf{E}^\omega\Big)\cdot\bm{\psi}\, ds
=\int_{\partial\Omega}\Big(\bm{\nu}\times\mathbf{E}^\omega\Big)\cdot\Big(\bm{\nu}\times\big(\bm{\nu}\times(\nabla\times\mathbf{F}) \big) \Big)\, ds\\
=& \int_{\partial\Omega}\Big(\bm{\nu}\times(\nabla\times\mathbf{F})\Big)\cdot\mathbf{E}^\omega\ ds-\int_{\partial\Omega}\Big(\bm{\nu}\times(\nabla\times\mathbf{E}^\omega) \Big)\cdot \mathbf{F}\ ds\\
=& \int_{\Omega}\big(\nabla\times\nabla\times\mathbf{F} \big)\cdot\mathbf{E}^\omega\, dV-\int_{\Omega}\big(\nabla\times\nabla\times\mathbf{E}^\omega \big)\cdot\mathbf{F}\, dV. 
\end{split}
\end{equation}
Using the fact that in $\Omega\backslash\overline{\Sigma}$, 
\begin{equation*}
\begin{split}
&\nabla\times\mathbf{E}^\omega-\mathrm{i}\omega\tau\beta_0\mathbf{H}^\omega=0,\\
&\nabla\times\mathbf{H}^\omega+\mathrm{i}\omega\tau^{-1}\epsilon_c\mathbf{E}^\omega=\tau^{-1}\sigma_c\mathbf{E}^\omega,
\end{split}\ \ \text{in}\ \ \Omega\backslash\overline{\Sigma}, 
\end{equation*}
one has by direct verifications that
\begin{equation}\label{eq:l25}
\nabla\times\nabla\times\mathbf{E}^\omega=\mathrm{i}\omega\beta_0(\sigma_c-\omega\epsilon_c)\mathbf{E}^\omega\quad\mbox{in}\ \ \Omega\backslash\overline{\Sigma}. 
\end{equation}
Plugging \eqref{eq:l25} into \eqref{eq:l23}, one then has
\begin{equation}\label{eq:l126}
\begin{split}
&\left|\int_{\partial\Omega}\Big(\bm{\nu}\times\mathbf{E}^\omega\Big)\cdot\bm{\psi}\, ds \right|\\
=&\left| \int_{\Omega}\big(\nabla\times\nabla\times\mathbf{F} \big)\cdot\mathbf{E}^\omega\, dV-\mathrm{i}\omega\beta_0\int_{\Omega}(\sigma_c-\omega\epsilon_c)\mathbf{E}^\omega\cdot \mathbf{F}\, dV \right|\\
\leq & C\|\mathbf{F}\|_{H^2(\mathrm{curl}, \Omega)}\|\mathbf{E}^\omega\|_{L^2(\Omega;\mathbb{C}^3)}\\
\leq & C \|\mathbf{E}^\omega\|_{L^2(\Omega;\mathbb{C}^3)}\|\bm{\psi}\|_{TH^{-1/2}_{\mathrm{Curl}}(\partial\Omega)},
\end{split}
\end{equation}
where $C$ depends only on $\Omega$ and $\alpha_0, \beta_0, \omega$. Finally, by combining \eqref{eq:l122}
and \eqref{eq:l126}, one immediately has \eqref{eq:l121}. 

The proof is complete. 
\end{proof}

We are in a position to complete the proof of Theorem~\ref{thm:nonscattering1}. First, by Lemmas~\ref{lem:l11} and \ref{lem:l12}, we have
\begin{equation}\label{eq:l127}
\begin{split}
&\|\bm{\nu}\times\mathbf{E}^\omega\|_{TH_{\mathrm{Div}}^{-1/2}(\partial \Omega)}\\
\leq & C\tau^{1/2}\|\mathbf{E}^{\kappa_\omega}_a\|_{H(\text{{curl}},\Omega)}+C\tau^{1/2}\|\mathbf{E}_{\mathbf{a}}^{\kappa_\omega}\|^{1/2}_{H(\mathrm{curl}, \Omega)}\Big\|\bm{\nu}\times \mathbf{E}^{s,\omega}\Big\|^{1/2}_{TH_{\mathrm{Div}}^{-1/2}(\partial\Omega)}\\
&+C\tau^{1/2}\|\mathbf{H}_{\mathbf{a}}^{\kappa_\omega}\|^{1/2}_{H(\mathrm{curl}, \Omega)}\Big\|\bm{\nu}\times \mathbf{E}^{s,\omega}\Big\|^{1/2}_{TH_{\mathrm{Div}}^{-1/2}(\partial\Omega)}+C\tau^{1/2}\Big\|\bm{\nu}\times \mathbf{E}^{s,\omega}\Big\|_{TH_{\mathrm{Div}}^{-1/2}(\partial\Omega)},\\
\leq & 2C\tau^{1/2}\|\mathbf{E}^{\kappa_\omega}_a\|_{H(\text{{curl}},\Omega)}+\frac{3C}{2}\tau^{1/2}\Big\|\bm{\nu}\times \mathbf{E}^{s,\omega}\Big\|_{TH_{\mathrm{Div}}^{-1/2}(\partial\Omega)}\\
&+C\tau^{1/2}\|\mathbf{H}_{\mathbf{a}}^{\kappa_\omega}\|_{H(\mathrm{curl}, \Omega)}. 
\end{split}
\end{equation} 
Then for sufficiently small $\tau>0$ such that $3C\tau^{1/2}\leq 1$, we readily deduce from \eqref{eq:l127} that
\begin{equation}\label{eq:l128}
\|\bm{\nu}\times\mathbf{E}^\omega\|_{TH_{\mathrm{Div}}^{-1/2}(\partial \Omega)}\leq C\tau^{1/2}\big(\|\mathbf{E}_{\mathbf{a}}^{\kappa_\omega}\|_{H^1(\mathrm{curl}, \Omega)}+\|\mathbf{H}_{\mathbf{a}}^{\kappa_\omega}\|_{H(\mathrm{curl}, \Omega)}\big). 
\end{equation}
Since
\[
(\mathbf{E}^{i,\omega},\mathbf{H}^{i,\omega})=(\mathbf{E}_{\mathbf{a}}^{\kappa_\omega},\mathbf{H}_{\mathbf{a}}^{\kappa_\omega})\in\mathcal{H}_\varepsilon(\mathbf{Y}),
\]
one clearly has
\begin{equation}\label{eq:l129}
\|\bm{\nu}\times\mathbf{E}^{i,\omega}\|_{TH_{\mathrm{Div}}^{-1/2}(\partial\Omega)}=\|\bm{\nu}\times\mathbf{E}_{\mathbf{a}}^{\kappa_\omega}\|_{TH_{\mathrm{Div}}^{-1/2}(\partial\Omega)}\leq \varepsilon. 
\end{equation}
Finally, by \eqref{eq:l128} and \eqref{eq:l129}, we have
\begin{equation}\label{eq:l130}
\begin{split}
\|\bm{\nu}\times\mathbf{E}^{s,\omega}\|_{TH_{\mathrm{Div}}^{-1/2}(\partial\Omega)}\leq & \|\bm{\nu}\times\mathbf{E}^\omega\|_{TH_{\mathrm{Div}}^{-1/2}(\partial\Omega)}+\|\bm{\nu}\times\mathbf{E}^{i,\omega}\|_{TH_{\mathrm{Div}}^{-1/2}(\partial\Omega)}\\
\leq & C\tau^{1/2}\big(\|\mathbf{E}_{\mathbf{a}}^{\kappa_\omega}\|_{H(\mathrm{curl}, \Omega)}+\|\mathbf{H}_{\mathbf{a}}^{\kappa_\omega}\|_{H(\mathrm{curl}, \Omega)}\big)+\varepsilon,
\end{split}
\end{equation}
which readily implies \eqref{eq:est1}.

The proof is complete. \hfill $\Box$

\subsection{Proof of Theorem~\ref{thm:nonscattering2}}

The proof follows from a similar argument to that for Theorem~\ref{thm:nonscattering1}. We shall give the necessary modifications in what follows. 

By taking the real and imaginary parts of both sides of \eqref{eq:pr12}, respectively, one has
\begin{align}
& \int_{\Omega}\sigma\mathbf{E}^\omega\cdot\overline{\mathbf{E}^\omega}\, dV=-\Re \int_{\partial\Omega}(\bm{\nu}\times\overline{\mathbf{E}^\omega})\cdot\mathbf{H}^\omega\, ds; \label{eq:ima1}\\
&\omega \int_{\Omega}\mu\overline{\mathbf{H}^\omega}\cdot\mathbf{H}^\omega\, dV=\omega\int_{\Omega}\epsilon\overline{\mathbf{E}^\omega}\cdot\mathbf{E}^\omega\, dV-\Im\int_{\partial\Omega}(\bm{\nu}\times\overline{\mathbf{E}^\omega})\cdot\mathbf{H}^\omega\, ds. \label{eq:rea1}
\end{align}
Using the specific form of the EM parameters in \eqref{eq:emp12}, and their ellipticity, one has from \eqref{eq:ima1} that
\begin{equation}\label{eq:ima12}
\alpha_0\int_{\Sigma}\big|\mathbf{E}^\omega \big|^2\, dV+\alpha_0\tau\int_{\Omega\backslash\overline{\Sigma}}\big|\mathbf{E}^\omega\big|^2\, dV\leq \Big|\Re\int_{\partial\Omega}(\bm{\nu}\times\overline{\mathbf{E}^\omega})\cdot\mathbf{H}^\omega\, ds \Big|,
\end{equation}
and from \eqref{eq:rea1} that
\begin{equation}\label{eq:rea12}
\begin{split}
\alpha_0\tau^{-1}\int_{\Omega\backslash\overline{\Sigma}}\big|\mathbf{H}^\omega \big|^2\, dV\leq &\alpha_0^{-1}\omega\int_{\Sigma}\big|\mathbf{E}^\omega\big|^2\, dV+\alpha_0^{-1}\tau\omega\int_{\Omega\backslash\overline{\Sigma}}\big|\mathbf{E}^\omega \big|^2\, dV\\
& +\Big|\Im\int_{\partial\Omega}(\bm{\nu}\times\overline{\mathbf{E}^\omega})\cdot\mathbf{H}^\omega\, ds \Big|
\end{split}
\end{equation}
By combining \eqref{eq:ima12} and \eqref{eq:rea12}, and using a similar argument in deriving \eqref{eq:l11}, one can show a similar estimate to \eqref{eq:l11} for $\|\mathbf{H}^\omega\|^2_{L^2(\Omega\backslash\overline{\Sigma};\mathbb{C}^3)}$, but involving $\bm{\nu}\times\mathbf{H}^{s,\omega}\big|_{\partial\Omega}$ in the RHS of the inequality. It is remarked that the estimate here depends on $\epsilon_b, \sigma_b$, but independent of $\mu_b$. Next, by a completely similar argument in deriving \eqref{eq:l122}, along with the use of the specific form of the EM parameters given in \eqref{eq:emp12}, one can show that 
\begin{equation}\label{eq:eh1}
\|\bm{\nu}\times\mathbf{H}^\omega\|_{TH_{\mathrm{Div}}^{-1/2}(\partial\Omega)}\leq C\|\mathbf{H}^\omega\|_{L^2(\Omega\backslash\overline{\Sigma}; \mathbb{C}^3)},
\end{equation}
By combining \eqref{eq:ima12}, \eqref{eq:rea12} and \eqref{eq:eh1}, and following a similar argument to that in deriving \eqref{eq:l130}, one can show that
\begin{equation*}
\|\bm{\nu}\times\mathbf{H}^{s,\omega}\|_{TH_{\mathrm{Div}}^{-1/2}(\partial\Omega)} \leq C\tau^{1/2}\big(\|\mathbf{E}_{\mathbf{a}}^{\kappa_\omega}\|_{H(\mathrm{curl}, \Omega)}+\|\mathbf{H}_{\mathbf{a}}^{\kappa_\omega}\|_{H(\mathrm{curl}, \Omega)}\big)+\varepsilon,
\end{equation*}
which readily implies \eqref{eq:est12}.

The proof is complete. \hfill $\Box$

\subsection{Proof of Theorem~\ref{thm:nonscattering3}}

Since $\omega\in\mathcal{T}(\Sigma;\epsilon_b,\mu_b)$ and $(\mathbf{E}^{i,\omega},\mathbf{H}^{i,\omega})=(\mathbf{E}_{\mathbf{a}}^{\kappa_\omega},\mathbf{H}_{\mathbf{a}}^{\kappa_\omega})\in\mathcal{H}_\varepsilon(\mathbf{T}_\omega)$, we let $(\mathbf{E}^{t,\omega},\mathbf{H}^{t,\omega})$ and $(\mathbf{U}^\omega,\mathbf{V}^\omega)\in \mathbf{T}_\omega$ be the interior transmission eigenfunctions corresponding to $\omega$ such that 
\begin{equation}\label{eq:pr31}
\|\mathbf{E}^{i,\omega}-\mathbf{U}^\omega\|_{H(\text{curl},\Sigma)}\leq \varepsilon \ \ \mbox{and}\ \  \|\mathbf{H}^{i,\omega}-\mathbf{V}^\omega\|_{H(\text{curl},\Sigma)}\leq \varepsilon. 
\end{equation}
Using \eqref{eq:total}, one readily has that
\begin{equation}\label{eq:pr32}
\begin{cases}
&\nabla\times\mathbf{E}^\omega-\mathrm{i}\omega\mu_b\mathbf{H}^\omega=0\quad\text{in}\ \ \Sigma,\medskip\\
&\nabla\times\mathbf{H}^\omega+\mathrm{i}\omega\epsilon_b\mathbf{E}^\omega=0\quad\, \text{in}\ \ \Sigma,\medskip\\
&\bm{\nu}\times\mathbf{E}^\omega=\bm{\nu}\times\mathbf{E}^{s,\omega}+\bm{\nu}\times\mathbf{E}^{i,\omega}\quad\ \text{on}\ \ \partial\Sigma,\medskip\\
&\bm{\nu}\times\mathbf{H}^\omega=\bm{\nu}\times\mathbf{H}^{s,\omega}+\bm{\nu}\times\mathbf{H}^{i,\omega}\quad\text{on}\ \ \partial\Sigma;
\end{cases}
\end{equation}
and
\begin{equation}\label{eq:pr33}
\begin{cases}
& \nabla\times\mathbf{E}^{s,\omega}-\mathrm{i}\omega\mu_\infty\mathbf{H}^{s,\omega}=0\ \ \ \text{in}\ \
 \mathbb{R}^3\backslash\overline{\Sigma},\medskip\\
&\nabla\times\mathbf{H}^{s,\omega}+\mathrm{i}\omega\epsilon_\infty\mathbf{E}^{s,\omega}=0\ \ \ \text{in}\ \ \mathbb{R}^3\backslash\overline{\Sigma},\medskip\\
&\displaystyle{\lim_{\|{\mathbf{x}}\|\rightarrow+\infty}\big(\mu_\infty^{1/2}\mathbf{H}^{s,\omega}(\mathbf{x})\times \mathbf{x}-\|{\mathbf{x}}\| \epsilon_\infty^{1/2}\mathbf{E}^{s,\omega}(\mathbf{x})\big)=0. }
\end{cases}
\end{equation}
By subtracting \eqref{eq:gitp} from \eqref{eq:pr32}, one further has
\begin{equation}\label{eq:pr34}
\begin{cases}
&\nabla\times(\mathbf{E}^\omega-\mathbf{E}^{t,\omega})-\mathrm{i}\omega\mu_b(\mathbf{H}^\omega-\mathbf{H}^{t,\omega})=0\quad\text{in}\ \ \Sigma,\medskip\\
&\nabla\times(\mathbf{H}^\omega-\mathbf{H}^{t,\omega})+\mathrm{i}\omega\epsilon_b(\mathbf{E}^\omega-\mathbf{E}^{t,\omega})=0\quad\, \text{in}\ \ \Sigma,\medskip\\
&\bm{\nu}\times(\mathbf{E}^\omega-\mathbf{E}^{t,\omega})=\bm{\nu}\times\mathbf{E}^{s,\omega}+\bm{\nu}\times(\mathbf{E}^{i,\omega}-\mathbf{E}^{t,\omega})\quad\ \text{on}\ \ \partial\Sigma,\medskip\\
&\bm{\nu}\times(\mathbf{H}^\omega-\mathbf{H}^{t,\omega})=\bm{\nu}\times\mathbf{H}^{s,\omega}+\bm{\nu}\times(\mathbf{H}^{i,\omega}-\mathbf{H}^{t,\omega})\quad\text{on}\ \ \partial\Sigma;
\end{cases}
\end{equation}
Next, using the transmission condition on $\partial\Sigma$, we have 
\begin{align}
\bm{\nu}\times(\mathbf{E}^\omega-\mathbf{E}^{t,\omega})=& \bm{\nu}\times\mathbf{E}^{s,\omega}+\bm{\nu}\times(\mathbf{E}^{i,\omega}-\mathbf{E}^{t,\omega})\nonumber\\
=&\bm{\nu}\times\mathbf{E}^{s,\omega}+\bm{\nu}\times(\mathbf{E}^{i,\omega}-\mathbf{U}^\omega),\label{eq:pr35}\\
\bm{\nu}\times(\mathbf{H}^\omega-\mathbf{H}^{t,\omega})=&\bm{\nu}\times\mathbf{H}^{s,\omega}+\bm{\nu}\times(\mathbf{H}^{i,\omega}-\mathbf{H}^{t,\omega})\nonumber\\
=&\bm{\nu}\times\mathbf{H}^{s,\omega}+\bm{\nu}\times(\mathbf{H}^{i,\omega}-\mathbf{V}^{\omega}).\label{eq:pr36}
\end{align}
Clearly, by virtue of \eqref{eq:pr34}, there holds
\begin{equation}\label{eq:pr37}
\bm{\nu}\times(\mathbf{H}^{\omega}-\mathbf{H}^{t,\omega})=\bm{\nu}\times(\mathbf{E}^\omega-\mathbf{E}^{t,\omega})\big|_{\partial\Sigma}.
\end{equation}
By \eqref{eq:pr35},\eqref{eq:pr36} and \eqref{eq:pr37}, one can deduce as follows,
\begin{equation*}
\begin{split}
\bm{\nu}\times\mathbf{H}^{s,\omega}=&\Lambda_{\Sigma,\omega}^i\big(\bm{\nu}\times(\mathbf{E}^{\omega}-\mathbf{E}^{t,\omega})\big)-\bm{\nu}\times (\mathbf{H}^{i,\omega}-\mathbf{V}^\omega),\\
=&\Lambda_{\Sigma,\omega}^i(\bm{\nu}\times\mathbf{E}^{s,\omega})+\Lambda_{\Sigma,\omega}^i\big(\bm{\nu}\times(\mathbf{E}^{i,\omega}-\mathbf{U}^\omega)\big)-\bm{\nu}\times (\mathbf{H}^{i,\omega}-\mathbf{V}^\omega),\\
=&\Lambda_{\Sigma,\omega}^i\circ\big(\Lambda_{\Sigma,\omega}^{o}\big)^{-1}(\bm{\nu}\times\mathbf{H}^{s,\omega})+\Lambda_{\Sigma,\omega}^i\big(\bm{\nu}\times(\mathbf{E}^{i,\omega}-\mathbf{U}^\omega)\big)\\
&-\bm{\nu}\times (\mathbf{H}^{i,\omega}-\mathbf{V}^\omega),
\end{split}
\end{equation*}
which then gives
\begin{equation}\label{eq:pr39}
\begin{split}
&\Big(\mathbf{I}-\Lambda_{\Sigma,\omega}^i\circ\big(\Lambda_{\Sigma,\omega}^{o}\Big)^{-1}\big)(\bm{\nu}\times\mathbf{H}^\omega)\\
=&\Lambda_{\Sigma,\omega}^i\big(\bm{\nu}\times(\mathbf{E}^{i,\omega}-\mathbf{U}^\omega)\big)-\bm{\nu}\times (\mathbf{H}^{i,\omega}-\mathbf{V}^\omega).
\end{split}
\end{equation}
Using the non-transparency condition and \eqref{eq:pr31}, one readily deduces from \eqref{eq:pr39} that\begin{equation}\label{eq:pr40}
\|\bm{\nu}\times\mathbf{H}^\omega\|_{TH_{\text{Div}}^{-1/2}(\partial\Sigma)}\leq C\varepsilon,
\end{equation}
where $C$ depends on $\omega, \epsilon_\infty, \mu_\infty$ and $\Sigma$. Finally, by applying \eqref{eq:pr40} to \eqref{eq:pr33}, one immediately has \eqref{eq:est123}.

The proof is complete. \hfill $\Box$

\section*{Acknowledgement}

The work was supported by the FRG grants from Hong Kong Baptist University, Hong Kong RGC General Research Funds, 12302415 and 405513, and the NSF grant of China, No. 11371115.

\end{document}